\documentclass[onefignum,onetabnum]{siamonline250106}

\ifpdf
\hypersetup{
  pdftitle={Random Reshuffling for SGLD},
  pdfauthor={Luke Shaw and Peter A. Whalley}
}
\fi

\usepackage{amsfonts}
\usepackage{graphicx}
\ifpdf
  \DeclareGraphicsExtensions{.eps,.pdf,.png,.jpg}
\else
  \DeclareGraphicsExtensions{.eps}
\fi

\usepackage{enumitem}
\setlist[enumerate]{leftmargin=.5in}
\setlist[itemize]{leftmargin=.5in}

\newsiamremark{remark}{Remark}
\newsiamremark{hypothesis}{Hypothesis}
\crefname{hypothesis}{Hypothesis}{Hypotheses}
\newsiamthm{claim}{Claim}

\headers{Random Reshuffling for SGLD}{Luke Shaw and Peter A. Whalley}

\title{Random Reshuffling for Stochastic Gradient Langevin Dynamics\thanks{Submitted to the editors \today.
\funding{LS acknowledges support from the Ministerio de Ciencia e Innovaci\'{o}n (Spain) through project PID2022-136585NB-C21, MCIN/AEI/10.13039/501100011033/FEDER, UE.}}}

\author{Luke Shaw\thanks{Departament de Matem\`{a}tiques and IMAC, Universitat Jaume I, 12071-Castell\'{o}n de la Plana, Spain
  (\email{shaw@uji.es}).}
\and Peter A. Whalley\thanks{Seminar for Statistics, Department of Mathematics, ETH Z{\"u}rich, Switzerland
  (\email{peter.whalley@math.ethz.ch})}}

\usepackage{amsopn}

\graphicspath{{figs/}}

\usepackage{subcaption} 
\usepackage{resizegather}
\usepackage{algorithm}
\usepackage{algpseudocode}
\usepackage{dsfont}
\usepackage{float}

\usepackage{todonotes}

\newcommand{\E}{\mathbb{E}}
\newcommand{\V}{\mathbb{V}}
\newcommand{\R}{\mathbb{R}}
\newcommand*{\horzbar}{\rule[.5ex]{2.5ex}{0.5pt}}

\DeclareMathOperator*{\argmax}{argmax}
\DeclareMathOperator*{\argmin}{argmin}

\newcommand{\mycomment}[1]{}
\newtheorem{assumption}{Assumption} 
\newtheorem{subassumption}{}[assumption]

\begin{document}

\maketitle
\begin{abstract}
    We examine the use of different randomisation strategies for one of the most popular stochastic gradient algorithms used in sampling, Stochastic Gradient Langevin Dynamics, based on first-order (or overdamped) Langevin dynamics. Conventionally, this algorithm is combined with a specific stochastic gradient strategy, called Robbins-Monro. In this work, we study an alternative strategy, Random Reshuffling, and show convincingly that it leads to improved performance via: a) a proof of reduced bias in the Wasserstein metric for strongly convex, gradient Lipschitz potentials; b) an analytical demonstration of reduced bias for a Gaussian model problem; and c) an empirical demonstration of reduced bias in numerical experiments for some logistic regression problems.
    This is especially important since Random Reshuffling is typically more efficient due to memory access and cache reasons. Such acceleration for the Random Reshuffling policy is familiar from the optimisation literature on stochastic gradient descent. 
\end{abstract}



\section{Introduction}
Within the field of machine learning, one way to train a model is to treat the problem as an optimisation problem (e.g. Maximum Likelihood Estimation, MLE), and to find the optimum of a function $F:\R^d\to\R$ using an optimiser, for example a (full-)gradient-based optimiser such as Gradient Descent (GD). However, models typically (and increasingly) learn from very large datasets, so that such optimisers require a computationally infeasible number of operations over the whole dataset. This has led to the substitution of such `full-gradient' optimisers by stochastic gradient optimisers which use gradients calculated over manageable subsets of data called `batches' such as stochastic gradient descent (SGD) \cite{Robbins1951}. This comes with a penalty however, since the use of stochastic gradients implies that the algorithm, operated with fixed step size $h$, asymptotically remains some distance from the true minimiser of the full gradient, an error which we call the `stochastic gradient bias' \cite{Bach2011}.

The search for better learning has also caused model complexity to grow with increasing dataset size, so that the number of parameters remains large relative to the amount of data, leading to significant uncertainty about optimal parameters. Hence, Bayesian learning approaches, which allow this uncertainty to be quantified as well as leading to models which are more resistant to overfitting, are desirable \cite{welling2011bayesian}.
Bayesian learning then amounts to an instance of the general `sampling problem': the problem of generating samples from a probability measure $\pi_*$ with density $p_*\propto\exp(-F)$, which occurs in other fields such as molecular dynamics \cite{Neal2011,LeimkuhlerBook2015,LeMa13}.

Sampling from this distribution (using e.g. a gradient-based MCMC sampler) would also require computations over the whole dataset, and so it is natural to try to solve the sampling problem via stochastic gradient samplers such as SGLD \cite{welling2011bayesian,Vollmer2016,Ahn2012}, just as one solves the optimisation problem via stochastic gradient optimisers such as SGD. In the search for better stochastic gradient methods one could of course focus on improving the underlying (full-)gradient-based method - e.g. shifting from gradient descent to a momentum-based optimiser \cite{Polyak1964} for optimisation, or from a first-order Langevin sampler to a second-order Langevin sampler or related method for sampling \cite{Chen2014,LeimkuhlerBook2015,Ma2021,Ding2014,LeMa13}. Other modifications of this type are also possible (see \cite{Baker2019,Brosse2018,Chatterji2018} and for a general survey \cite{Nemeth2021}). However, another approach is to fix the sampler/optimiser, and vary the randomisation strategy used to generate the stochastic gradient at each iteration. Within optimisation, this has proven to be an extremely productive approach, both theoretically and experimentally, where the two main strategies are Robbins-Monro (RM) and Random Reshuffling (RR), with the latter typically preferred \cite{Recht2012,Sun2020,Bottou2009}{:
\begin{itemize}
    \item Robbins-Monro: a subset of the data is randomly sampled from the full dataset at each iteration of the algorithm and used to calculate the stochastic gradient.
    \item Random Reshuffling: the full dataset is shuffled and partitioned into $R$ subsets, with the subsets cycled through over $R$ iterations. This reshuffling, repartitioning and cycling is repeated for as many epochs of $R$ iterations as required.
\end{itemize}
}
It is now well-established that RR is not only more computationally efficient (due to caching and memory access \cite{Bengio2012}) but incurs a reduced stochastic gradient bias relative to RM when both are applied to SGD (giving SGD-RR and SGD-RM respectively) \cite{Mishchenko2020,Cha2023}. Since one can apply the same randomisation strategies (RR and RM) to stochastic gradient \emph{samplers}, one is led to ask what relation exists between the two strategies for sampling - currently only the RM strategy has been analysed or even been examined experimentally\footnote{Except for \cite{paulin2024sampling}, which considers second-order Langevin dynamics.} despite the well-recognised connection between optimisation and sampling \cite{Dalalyan2017further,Durmus2019}. In this work we consider the Euler-Maruyama sampler for the overdamped Langevin equation (called `ULA' \cite{Roberts1996}), which when used with a stochastic gradient is known as the Stochastic Gradient Langevin Dynamics sampler (SGLD) \cite{welling2011bayesian}. For this sampler, we show that, just as for optimisation, RR (SGLD-RR) improves on RM (SGLD-RM), in three ways:
\begin{itemize}
    \item Under standard assumptions on $\pi_*$, we achieve state of the art non-asymptotic guarantees for SGLD-RR which improve upon known results for SGLD-RM in \cref{sec:Convergence}. In particular, to reach an accuracy of $\epsilon > 0$ in Wasserstein-2 distance our results for SGLD-RR require $\mathcal{O}(1/\epsilon)$ steps compared to the best-known bound of $\mathcal{O}(1/\epsilon^{2})$ for SGLD-RM.
    \item We show analytically for a toy Gaussian problem that the bias of SGLD-RR is indeed less than that of SGLD-RM in \cref{sec:ModelProblem}.
    \item We show numerically for some logistic regression problems from \cite{Casas2022} that, empirically, the bias of SGLD-RR is indeed less than that of SGLD-RM in \cref{sec:Experiments}.
\end{itemize}

\section{Preliminaries}
In the standard machine learning problem, one assumes a model param-etrised by some parameters $X\in\R^d$ for a dataset $Y\equiv\{y_i\}_{i=1}^N$. This parametrisation is encoded via a likelihood per observation $\ell(\cdot;y_i):\mathbb{R}^{d} \to \mathbb{R}$, so that the overall likelihood of the dataset is $\mathcal{L}: \mathbb{R}^{d} \to \mathbb{R}$ defined by $\mathcal{L}(\cdot;Y)=\prod_{i=1}^N\ell(\cdot;y_i)$. One could then simply perform MLE using some optimisation algorithm to arrive at a more or less `good' parameter value.

On the other hand, within a Bayesian framework, one incorporates prior beliefs about the `true' parameters (e.g. sparsity) via a prior with density $\nu_0$ to give a posterior $\pi_*$ with density defined for $X\in \mathbb{R}^{d}$ by $p_*(X)\propto \nu_0(X)\times\mathcal{L}(X;Y)$, i.e.,
\begin{equation}\label{eq:posterior}
p_*(X)=\nu_0(X)\prod_{i=1}^{N}\ell(X;y_i)\propto\exp\left(-\frac{1}{N}\sum_{i=1}^Ng_i(X)\right)=\exp\left(-F(X)\right),
\end{equation}
where $g_i(X)=-N\log(\ell(X;y_i))-\log(\nu_{0}(X))$.

The most likely parameter value is then the Maximum A Posteriori (MAP) parameter estimate i.e. the minimum of the negative log-posterior (assuming there is a unique minimum), equal to $F$ up to an unimportant constant,
\begin{equation}
X_*=\argmax_{X\in \mathbb{R}^{d}} p_*(X)= \argmin_{X\in \mathbb{R}^{d}} F(X).
\end{equation}
This can be obtained via a gradient-based optimisation algorithm such as GD which, along with step size $h$ and starting value $x_0$, is defined by the update rule
\begin{equation}\label{eq:GD}\tag{GD}
    x_{k+1}\gets x_k-h\nabla F(x_k),
\end{equation}
which under appropriate assumptions on $h$ and $F$ converges to the MAP $\lim_{K\to\infty}x_K=X_*$ \cite{NesterovBook}. However, for $N$ large, this is too costly per iteration, and so one replaces the gradient $\nabla F$ in \cref{eq:GD} with a \emph{stochastic gradient} $\widehat{\nabla F}$, which is an unbiased estimator of $\nabla F$ with some reduced cost $n\ll N$. This gives the update rule for the SGD algorithm
\begin{equation}\label{eq:SGD}\tag{SGD}
\begin{aligned}
x_{k+1}&\gets x_k-h\widehat{\nabla F}(x_k),
\end{aligned}
\end{equation}
naturally incurring some error, so that $\lim_{K\to\infty}x_K$ is only `close' to $X_*$ (see \cite{Gower2019,Bottou2004}).

However, the real strength of the Bayesian approach is that using  Bayesian learning (e.g for Bayesian Neural Networks) one may quantify uncertainty about the parameter estimates and also avoid overfitting (using ensemble models as in e.g. \cite{Ahn2012}) via sampling from the posterior $\pi_*$ \cite{NealBNNs,Mackay1992}.  
Since there is always some uncertainty regarding the MAP with respect to changes in e.g. slightly different data or prior beliefs, especially in the case of the very large models with huge numbers of parameters that are popular today, this is a rather powerful advantage. 

This sampling can be accomplished via taking the overdamped or first-order Langevin equation for $X_t\in\mathbb{R}^{d}$ associated to the potential $F$
\begin{equation}\label{eq:ODLangevin}
    dX_t=-\nabla F(X_t)dt+\sqrt{2}dW_t,
\end{equation}
where $W_{t}$ is a $d$-dimensional standard Brownian motion. Under mild conditions, it can be shown that \cref{eq:ODLangevin} has invariant measure $\pi_{*}$ associated with the density $p_*\propto\exp(-F)$ (see \cite[Proposition 4.2]{pavliotis2014stochastic}). In practice these dynamics cannot be solved exactly, so one needs to discretise the dynamics. Discretising it using the Euler-Maruyama scheme gives the `Unadjusted Langevin Algorithm' (ULA), defined by the update step \cite{Durmus2017,Roberts1996}
\begin{equation}\label{eq:GLD}\tag{ULA}
    x_{k+1}\gets x_k-h\nabla F(x_k)+\sqrt{2h}\xi_k,
\end{equation}
along with step size $h>0$ and initialisation $x_0\in \mathbb{R}^{d}$, and where $\xi_k\sim\mathcal{N}(0_{d},I_{d})$ are independent and identically distributed. Denote the distribution of samples generated after $k$ iterations of \cref{eq:GLD} as $\widetilde{\pi}_k^{[ULA]}$. Under appropriate assumptions, the limiting distribution of the samples $\lim_{k\to\infty}\widetilde{\pi}_k^{[ULA]}\equiv\widetilde{\pi}_{\infty}^{[ULA]}$ is $\mathcal{O}(h)$ close to the true distribution $\pi_*$ in Wasserstein distance \cite{Durmus2017,Durmus2019,Dalalyan2017further}.

However as was the case in the optimisation setting, when $N$ is large, one often chooses to use cheaper approximations of the gradient in  \cref{eq:GLD}, giving the SGLD algorithm update step \cite{welling2011bayesian}
\begin{equation}\label{eq:SGLD}\tag{SGLD}
\begin{aligned}
x_{k+1}&\gets x_k-h\widehat{\nabla F}(x_k)+\sqrt{2h}\xi_k,
\end{aligned}
\end{equation}
incurring an additional error so that the limiting distribution of the samples generated by \cref{eq:SGLD} $\widetilde{\pi}_{\infty}^{[SGLD]}$ is even further from $\pi_*$ than $\widetilde{\pi}_{\infty}^{[ULA]}$ is \cite{DalalyanSG}.

We note in passing that, even when not using the Bayesian framework and simply directly minimising a loss function $F:\R^d\to \R$, as is the case in the classic machine learning problem, sampling-type algorithms can still be useful - especially for non-convex potentials - since the samples concentrate around the minimiser of $F$ \cite{Zhang2023, Dalalyan2017further,DalalyanSG,Xu2018,Raginsky2017}.

\subsection{Randomisation Strategies}\label{sec:randomisation_strategies}
We now turn to how one generates the stochastic gradient estimator $\widehat{\nabla F}$ in \cref{eq:SGD,eq:SGLD}, following terminology from \cite{BertsekasBook,Mishchenko2020}. A review of results for randomisation strategies for optimisation with stochastic gradient descent can be found in \cite{Cha2023}.
Recall that $\nabla F$ is based on a finite sum
 \begin{equation}\label{eq:FiniteSum}
 \nabla F=\frac{1}{N}\sum_{i=1}^{N}\nabla g_i.
  \end{equation}
 The stochastic gradient optimisers and samplers \cref{eq:SGD,eq:SGLD} implemented with policies RR and RM may then be understood as special cases of the same algorithm \cref{alg:GenAlg}.
 One generates a matrix $\Omega_{n,N}(m)\in\R^{m\times n}$
 \begin{equation}\label{eq:OmegaMatrix}
     \Omega_{n,N}(m)=\begin{bmatrix}
         \omega_{11} & \omega_{12} & \ldots &  \omega_{1n}\\
         \omega_{21} & \ddots & \ldots &  \omega_{2n}\\
         \vdots & \vdots & \ddots &  \vdots\\
        \omega_{m1} & \ldots & \ldots &  \omega_{mn}
     \end{bmatrix}
     =\begin{bmatrix}
          \horzbar & \boldsymbol{\omega}_{1} &\horzbar \\
         \horzbar & \boldsymbol{\omega}_{2} &\horzbar\\
          & \vdots &\\
        \horzbar & \boldsymbol{\omega}_{m} &\horzbar\\
     \end{bmatrix},
 \end{equation}
 of $mn$ distinct entries $\omega_{ij}$ drawn without replacement\footnote{One could also sample with replacement \cite{Vollmer2016}.} from $\{1,2,\ldots,N\}$. We will often suppress the additional notation for brevity\footnote{{However, $\boldsymbol{\omega}$, with or without a subscript, will always refer to a vector of length $n$ drawn uniformly without replacement from $\{1,...,N\}$; $\omega$, with or without subscripts, will always refer to an integer from the set $\{1,\ldots, N\}$; and $\Omega$, with or without subscripts, will refer to a matrix, as defined above.}}, writing simply $\Omega$ in the following. 
 For the case $mn=N$, one has maximal $m=R=N/n$ (we assume $n$ divides $N$ exactly for simplicity), and the matrix exactly partitions the set of indices. A batch is then defined as the set of indices corresponding to an $\boldsymbol{\omega}_{i}=(\omega_{i1},\omega_{i2}, \ldots,\omega_{in})$, and gives a stochastic gradient via
 \begin{equation}\label{eq:stochgrad}
      \widehat{\nabla F}=\nabla f_{\boldsymbol{\omega}_i},\quad \nabla f_{\boldsymbol{\omega}_i}\equiv\frac{1}{n}\sum_{j=1}^{n}\nabla g_{\omega_{ij}}.
 \end{equation}
Note that no two batches in the same matrix then contain the same $\nabla g_i$. 

We may now define the aforementioned strategies for generating the stochastic gradient, RM and RR. Stochastic gradient algorithms are typically run for a total number of iterations which is an integer multiple of $R$,  $K=n_eR,n_e\in\mathbb{N}$. $n_e$ is then the \emph{number of epochs}, where an epoch is composed of $R$ iterations.

\paragraph{Robbins-Monro (RM)} For RM, one sets $m=1$ and the stochastic gradient is then generated at each step $k$ via sampling $\Omega\in\R^{1\times n}$ and then calculating $\widehat{\nabla F}$ via \cref{eq:stochgrad}.
In expectation then, after one epoch of $R$ iterations, the algorithm sees every $\nabla g_i$ once \cite{Robbins1951}.

\paragraph{Random Reshuffling (RR)}
In the case of RR, one sets $m=R$ first, samples $\Omega\in\R^{R\times n}$ and then calculates $\widehat{\nabla F}$  via \cref{eq:stochgrad} with $\boldsymbol{\omega}_1$. Over the succeeding $R-1$ steps one iterates through the remaining batches $\boldsymbol{\omega}_2,\ldots,\boldsymbol{\omega}_R$, before reaching the end of the epoch. One then resamples the matrix $\Omega$ and carries out another $R$ steps. Hence, after one epoch of $R$ iterations, the algorithm has seen every $\nabla g_i$ once and only once.
For an intuition as to why RR is superior to RM, see \cite[Exercise 2.10]{BertsekasBook}.

\begin{algorithm}[H]
    \textbf{Input}: $m,K,N,n,x_0,h$ \Comment{$m=1$ if RM, $m=R=N/n$ if RR}
    \caption{General Stochastic Gradient Algorithm}\label{alg:GenAlg}
    \begin{algorithmic}[1]
    \State $k\gets 0$
    \For{$l=1,\ldots, K/m$} \Comment{In practice $K=n_eR$, $n_e\in\mathbb{N}$}
        \State Sample $\Omega_{n,N}(m)=\{\boldsymbol{\omega}_i\}_{i=1}^m$ according to \cref{eq:OmegaMatrix} 
        \For{$i=1,\ldots,m$}
            \State Calculate $\widehat{\nabla F}\gets\nabla f_{\boldsymbol{\omega}_i}$ via \cref{eq:stochgrad}
            \State Generate $x_{k+1}$ using $h\widehat{\nabla F}(x_k)$ via \cref{eq:SGD} or \cref{eq:SGLD}
            \State $k\gets k+1$
        \EndFor
    \EndFor
    \State \Return{$(x_k)_{k=0}^K$}
    \end{algorithmic}
\end{algorithm}

\paragraph{Variance of $\widehat{\nabla F}$} As mentioned, the stochastic gradient estimator is unbiased\footnote{Note that for conditional expectations based on previous steps RM remains unbiased i.e. $\mathbb{E}[\widehat{\nabla F}(x_k)|x_k]=\nabla F(x_k)$. This is not the case for RR \cite{Mishchenko2020}.} for both RR and RM, $\mathbb{E}[\widehat{\nabla F}]=\nabla F$ but it has a variance, $\sigma^2_*$. This variance is the source of the stochastic gradient bias which affects both SGD and SGLD. We show presently that $\sigma_*^2$ grows like $R$ {(cf. \cite[Prop. 3.10]{Gower2019})}, if one assumes that the variance can in fact be bounded.
\begin{lemma}[$\sigma_*^2\propto R$]\label{lem:sigmaK}
The variance of the stochastic gradient estimator follows
\begin{align*}
\sigma_*^2\equiv \E_{X\sim q,\boldsymbol{\omega}}\left\|\nabla f_{\boldsymbol{\omega}}(X)-\nabla F(X)\right\|^2=\frac{R-1}{N-1}C_G,\\
\intertext{where {$\boldsymbol{\omega}$ is a random vector of length $n$ drawn uniformly without replacement from $\{1,...,N\}$}, }
C_G=N^{-1}\sum_{i=1}^N\mathbb{E}_{X\sim q}\|\nabla g_i(X)-\nabla F(X)\|^2<\infty,
\end{align*}
and $q=\pi_*$ for the SGLD variance and $q=\delta_{X_*}$ for the SGD variance.
\end{lemma}
\begin{proof}
Using that $\nabla f_{\boldsymbol{\omega}}=\frac{1}{n}\sum_{j=1}^n\nabla g_{\omega_{j}}$, {(where $\omega_{j}$ are the components of the vector $\boldsymbol{\omega}$)} one has 
$$\E_{X\sim q,\boldsymbol{\omega}}\left\|f_{\boldsymbol{\omega}}(X)-\nabla F(X)\right\|^2=\E_{X\sim q,\boldsymbol{\omega}}\left\|\frac{1}{n}\sum_{j=1}^{n}\nabla g_{\omega_{j}}(X)-\nabla F(X)\right\|^2.$$ 
Since $\mathbb{E}_{\omega_j}[\nabla g_{\omega_j}-\nabla F]=\frac{1}{N}\sum_{i=1}^N\nabla g_i-\nabla F=0$ {(as $\omega_j$ are unconditionally sampled from $\{1,\ldots, N\}$ uniformly)}, one may then apply the result \cite[Lemma 1]{Mishchenko2020} to give
\begin{align*}
\E_{X\sim q}\mathbb{E}_{\boldsymbol{\omega}}\left\|\frac{1}{n}\sum_{j=1}^{n}\nabla g_{\omega_{j}}(X)-\nabla F(X)\right\|^2&\overset{{\textnormal{\cite[Lemma 1]{Mishchenko2020}}}}{=}\frac{N-n}{n(N-1)}\E_{X\sim q,\omega_j}\left\|\nabla g_{\omega_{j}}(X)-\nabla F(X)\right\|^2\\
&=\frac{R-1}{N-1}C_G.
\end{align*}
\end{proof}

Although the result of \cref{lem:sigmaK} is illuminating, in order for the proofs in later sections to be valid, all that is required is boundedness of the variance, which we enforce via the following assumption.
\begin{assumption}[Moments of Stochastic Gradient]\label{assum:stochastic_gradient}
We assume that
$$
\sigma_*^2\equiv \E_{X\sim q,\boldsymbol{\omega}}\left\|\nabla f_{\boldsymbol{\omega}}(X)-\nabla F(X)\right\|^2<\infty,
$$
where $q=\pi_*$ for the SGLD variance and $q=\delta_{X_*}$ for the SGD variance.
\end{assumption}

\section{Convergence guarantees for stochastic gradient algorithms}\label{sec:Convergence}
For the optimisation problem, it was well-known for many years that SGD-RR typically performed much better than SGD-RM and this superiority has now been placed on a firm theoretical footing (see \cite{Mishchenko2020,Cha2023} and references), and so one can now state two companion theorems for optimisation with SGD-RM (\cref{thm:SGDRM}) and SGD-RR (\cref{thm:SGDRR}). 

The understanding of convergence of SGLD(-RM/-RR) is somewhat less mature, and, to the best of our knowledge, very little experimental or theoretical study has been made of the differences between RM and RR. The key results for convergence of SGLD-RM may be found in \cite{DalalyanSG}, and so one is able to provide the sampling analogue SGLD-RM (\cref{thm:SGLDRM}) to the optimisation case. In this work we prove the complementary theorem for SGLD-RR (\cref{thm:SGLDRR}), filling the gap suggested by the optimisation companion theorems.

\subsection{Assumptions} 

In the following we make strong assumptions on $F$ and $f_{\boldsymbol{\omega}}$ for all $\boldsymbol{\omega} \subset \{1,2,...,N \}$ such that $|\boldsymbol{\omega}| = n<N$. These strong assumptions allow for the quantitative convergence guarantees that are available in the optimisation and sampling literature.

\begin{definition}\label{def:Lsmooth}
A continuously differentiable $f:\mathbb{R}^{d} \to \mathbb{R}$ is $L$-smooth if  for all $x,y\in \mathbb{R}^{d}$ {, with $L>0$,}
$$
\lVert \nabla f(x) - \nabla f(y)\rVert \le L \lVert x-y\rVert.
$$
\end{definition}
\begin{definition}\label{def:muconvex}
A continuously differentiable $f:\mathbb{R}^{d} \to \mathbb{R}$ is $\mu$-strongly convex if for all $x,y \in \mathbb{R}^{d}$ {, with $\mu>0$,}
$$
f(x) \geq f(y)+\langle \nabla f(y),x-y\rangle + \frac{\mu}{2} \lVert x-y\rVert^2.
$$
Further, $f$ is {(merely)} convex if it is $0$-strongly convex.
\end{definition}

\begin{definition}\label{def:L1HessianLipschitz}
A twice continuously differentiable $f:\mathbb{R}^{d} \to \mathbb{R}$ has a $L_{1}$-smooth Hessian if for all $x,y\in \mathbb{R}^{d}${, with $L_1>0$,}
\[
\|\nabla^{2}f(x) - \nabla^{2}f(y)\| \leq L_{1}\|x-y\|.
\]
\end{definition}
\begin{remark}
    It is well known that if a function $f:\mathbb{R}^{d} \to \mathbb{R}$ is twice continuously differentiable, $L$-smooth and $\mu$-strongly convex then for all $x \in \mathbb{R}^{d}$
    \[
    \mu I_{d} \prec \nabla^{2}f(x) \prec L I_{d}.
    \]
\end{remark}

\begin{assumption}\label{assum:smoothness}
For some positive constants $\mu,L \in \mathbb{R}_{+}$ we assume the potential $F:\mathbb{R}^{d} \to \mathbb{R}$ is of the form \cref{eq:FiniteSum} and is continuously differentiable, $L$-smooth and $\mu$-strongly convex. Further assume that for any $\boldsymbol{\omega} \subset \{1,2,...,N \}$ such that $|\boldsymbol{\omega}| = n<N$, $f_{\boldsymbol{\omega}}:\mathbb{R}^{d}\to\mathbb{R}$ defined by \cref{eq:stochgrad} are continuously differentiable and $L$-smooth.
\begin{subassumption} 
In addition to the above, assume $f_{\boldsymbol{\omega}}:\mathbb{R}^{d}\to\mathbb{R}$ are convex.\label{assum:smoothnessA}
\end{subassumption}
\begin{subassumption}
    In addition to the above, assume $f_{\boldsymbol{\omega}}:\mathbb{R}^{d}\to\mathbb{R}$ are $\mu$-strongly convex.\label{assum:smoothnessB}
\end{subassumption}
\end{assumption}

\begin{assumption}\label{assum:additional_smoothness}
The potential $F:\mathbb{R}^{d} \to \mathbb{R}$ is twice continuously differentiable and has a $L_{1}$-smooth Hessian. 
\end{assumption}

\subsection{SGD}
If one seeks the minimiser $X_*\in \mathbb{R}^{d}$ of a function $F:\mathbb{R}^{d} \to \mathbb{R}$, it may be shown that the infinite-iteration final iterate $\lim_{K\to\infty}x_K$ of \cref{eq:SGD} converges to a point close to the minimiser, and that this point is closer to $X_*$ if RR is used instead of RM.
\begin{theorem}[SGD-RM]\label{thm:SGDRM}
For a function $F:\mathbb{R}^{d} \to \mathbb{R}$ satisfying \cref{assum:smoothness}\cref{assum:smoothnessA} with minimiser $X_{*} \in \mathbb{R}^{d}$ consider iterate $x_K$ for $K\in \mathbb{N}$ generated by SGD-RM via \cref{eq:SGD} with objective $F$, stepsize $0<h<1/(2L)$ and starting iterate $x_0\in\R^d$. Assume the stochastic gradients satisfy \cref{assum:stochastic_gradient} with constant $0<\sigma^{2}_{*}<\infty$  then
$$
\E\|x_K-X_*\|^2\leq(1-h\mu)^K\|x_0-X_*\|^2+2h\frac{\sigma_*^2}{\mu}.
$$
\begin{proof}
   Follows from \cite[Proof of Thm. 3.1]{Gower2019} or \cite[Thm. 5.8]{Garrigos2023}.
\end{proof}
\end{theorem}

\begin{theorem}[SGD-RR]\label{thm:SGDRR}
For a function $F:\mathbb{R}^{d} \to \mathbb{R}$ satisfying \cref{assum:smoothness}\cref{assum:smoothnessB} with minimiser $X_{*} \in \mathbb{R}^{d}$ consider iterate $x_K$ generated by an integer number of epochs of SGD-RR (i.e. $K=n_{e}R$, $n_e,R\in\mathbb{N}$) via \cref{eq:SGD} with objective $F$, stepsize $0<h<1/L$ and starting iterate $x_0\in\R^d$. Assume the stochastic gradients satisfy \cref{assum:stochastic_gradient} with constant $0<\sigma^{2}_{*}<\infty$ then
$$
\E\left\| x_K-X_*\right\|^2\leq(1-h\mu)^{K}\left\| x_0-X_*\right\|^2+h^2\frac{LR}{2\mu}\sigma_*^2.
$$
\begin{proof}
    Given in \cite[Thm. 1]{Mishchenko2020}.
\end{proof}
\end{theorem}
\begin{remark}
In the limit $n_e\to\infty$, with $K=n_eR$, the first term in each of the bounds in \cref{thm:SGDRM,thm:SGDRR} goes to 0 and one is left with the second term as a remainder. As shown by \cref{lem:sigmaK}, $\sigma_*^2\propto R$, and so one can think of the switch from SGD-RM to SGD-RR as decreasing the asymptotic bias from $\mathcal{O}(Rh)$ to $\mathcal{O}((Rh)^2)$.
\end{remark}

\paragraph{SGLD}
If one seeks to sample from the distribution $\pi_*$ associated to the potential $F$, it may be shown that one asymptotically samples, in the limit ${n_e\to\infty}$, from distributions that are closer to $\pi_*$ if RR is used instead of RM. In the succeeding we denote the distribution of samples generated after $k$ iterations of the SGLD scheme \cref{eq:SGLD} (whether by SGLD-RR or SGLD-RM) as $\widetilde{\pi}_k$.

\begin{theorem}[SGLD-RM]\label{thm:SGLDRM}
For a target measure $\pi_{*}$ with negative log-density $F:\mathbb{R}^{d} \to \mathbb{R}$ satisfying \cref{assum:smoothness}\cref{assum:smoothnessA} consider iterates $(x_{k})_{k\in\mathbb{N}}$ generated by SGLD-RM via \cref{eq:SGLD} with objective $F$, stepsize $0<h<2/(L+\mu)$ and $x_{0}\sim\widetilde{\pi}_0$. Assuming the stochastic gradients satisfy \cref{assum:stochastic_gradient} with constant $0<\sigma^{2}_{*}<\infty$ we have for $K\in \mathbb{N}$
\begin{equation}\label{eq:SGLDRM_noass3}     \mathcal{W}_{2}(\widetilde{\pi}_{K},\pi_*) \leq (1-h\mu)^{K}\mathcal{W}_{2}(\widetilde{\pi}_{0},\pi_*) + \sqrt{h}\frac{1.65L\sqrt{d}}{\mu}+\frac{\sigma_*\sqrt{h}}{\sqrt{\mu}}.
\end{equation}   
If in addition $F$ satisfies \cref{assum:additional_smoothness} then
\begin{equation}\label{eq:SGLDRM_ass3} 
    \mathcal{W}_{2}(\widetilde{\pi}_{K},\pi_*) \leq (1-h\mu)^{K}\mathcal{W}_{2}(\widetilde{\pi}_{0},\pi_*) + h\left(\frac{L_1d}{2\mu} +\frac{11L\sqrt{Ld}}{5\mu}\right)+\frac{\sigma_{*}\sqrt{h}}{\sqrt{\mu}}.
\end{equation}
\begin{proof}
    See\footnote{The bounds we show here are slightly looser than in the cited work - we have relaxed them for readability.} \cite[Thms. 4 and 5]{DalalyanSG}.
\end{proof}
\end{theorem}

\begin{theorem}[SGLD-RR]\label{thm:SGLDRR}
For a target measure $\pi_{*}$ with negative log-density $F:\mathbb{R}^{d} \to \mathbb{R}$ satisfying \cref{assum:smoothness}\cref{assum:smoothnessB} consider iterates $(x_{k})_{k\in\mathbb{N}}$ generated by SGLD-RR via \cref{eq:SGLD} with objective $F$, stepsize $0<h<1/L$ and $x_{0}\sim\widetilde{\pi}_0$. Assuming the stochastic gradients satisfy \cref{assum:stochastic_gradient} with constant $0<\sigma^{2}_{*}<\infty$ we have for $K\in \mathbb{N}$
\begin{equation}\label{eq:SGLDRR_noass3} 
        \mathcal{W}_{2}(\widetilde{\pi}_{K},\pi_*)\leq (1-h\mu)^{K}\mathcal{W}_{2}(\widetilde{\pi}_{0},\pi_*) + 240\frac{L}{\mu}\Bigg(h\sqrt{R}\sigma_{*}  + h(LR\sqrt{hd} + \sqrt{LRd})+ \sqrt{hd}\Bigg).
    \end{equation}
    If in addition $F$ satisfies \cref{assum:additional_smoothness} then
    \begin{equation}\label{eq:SGLDRR_ass3} 
        \mathcal{W}_{2}(\widetilde{\pi}_{K},\pi_*) \leq (1-h\mu)^{K}\mathcal{W}_{2}(\widetilde{\pi}_{0},\pi_*) + 240h\frac{L}{\mu}\Bigg(\sqrt{R}\sigma_{*}  + (LR\sqrt{hd} + \sqrt{LRd})+\left(\sqrt{Ld} +\frac{L_{1}}{L}d\right)\Bigg).
    \end{equation}
\begin{proof}
    See \cref{sec:proofs}.
\end{proof}
\end{theorem}

\begin{remark}
One has to be careful about interpreting \cref{thm:SGLDRM,thm:SGLDRR} since SGLD-RR does not have an invariant distribution (as is confirmed experimentally by the oscillation plots in \cref{fig:Exps}), although SGLD-RM does. However, {if we compare the results without \cref{assum:additional_smoothness}} one can see that in the limit $K\to\infty$, the first term in each of the bounds in \cref{eq:SGLDRM_noass3,eq:SGLDRR_noass3} goes to 0 and one is left with other terms as a remainder. 
As shown by \cref{lem:sigmaK}, $\sigma_*^2\propto R$, and hence ({in the case where \cref{assum:additional_smoothness} is not fulfilled) SGLD-RM asymptotically samples from a distribution that is $\mathcal{O}(h^{1/2}+(Rh)^{1/2})$-close to $\pi_*$, whereas SGLD-RR asymptotically samples from distributions that are instead $\mathcal{O}(h^{1/2}+Rh)$-close.}
\end{remark}

\begin{remark}
{If \cref{assum:additional_smoothness} holds, (cf. \cref{eq:SGLDRM_ass3,eq:SGLDRR_ass3}) we also have improved asymptotic convergence of $\mathcal{O}(h+Rh)$ for SGLD-RR  compared to $\mathcal{O}(h+(Rh)^{1/2})$ for SGLD-RM}. As a consequence, {for targets which fulfill \cref{assum:additional_smoothness}}, the number of steps $K$ to reach an error tolerance $\epsilon>0$ of the target measure scales as $\mathcal{O}(1/\epsilon)$ for SGLD-RR, which is a considerable improvement over $\mathcal{O}(1/\epsilon^{2})$ for SGLD-RM. One also may expect an improvement in the scaling in terms of dimension if stronger assumptions are made, as has been done in the full gradient setting in \cite{li2021sqrt}. 
\end{remark}

\begin{remark}{Note that the terms due to the stochastic gradient bias (involving $R$) are identical in the two bounds in \cref{thm:SGLDRR}, as \cref{assum:additional_smoothness} only alters the error contribution due to the discretisation bias. Even when the stochastic gradient bias disappears when using full gradients (i.e. \cref{eq:GLD}), it is not possible to improve on $\mathcal{O}(1/\epsilon)$ guarantees obtained under \cref{assum:additional_smoothness}.}
\end{remark}

\section{A Model Problem}\label{sec:ModelProblem}
In order to probe the tightness of the bounds in \cref{thm:SGLDRM,thm:SGLDRR}, we consider the application of SGLD-RR and SGLD-RM to a 1D sampling problem from \cite{Leimkuhler2016} for a dataset $Y=\{y_i\}_{i=1}^N$ with $y_i\in\R$ for $i=1,...,N$ and we define
\begin{equation}\label{eq:ModelProblem}
X\sim\mathcal{N}(\bar{y},\sigma^2/N),\quad\textnormal{where}\quad \quad \bar{y}\equiv \frac{1}{N}\sum_{i=1}^Ny_i. 
\end{equation}

Note that this sampling problem can be arrived at via considering $\{y_i\}_{i=1}^N$ independent and identically distributed random variables under a parametrised model $y_{\vert X}\sim\mathcal{N}(X,\sigma_y^2)$, and formulating a Bayesian inference problem for the parameter $X$, applying either a uniform prior as in \cite{Leimkuhler2016} or Gaussian prior (and rescaling) as in \cite{Vollmer2016}. Note also that $d$-dimensional versions of such problems which result in $X\sim\mathcal{N}(\boldsymbol{m},C)$ with $C\in\R^{d\times d}$ positive definite admit diagonalisation, giving uncoupled copies of the 1D case \cref{eq:ModelProblem}, reinforcing the relevance of this problem, since such diagonalisation commutes with the sampling algorithms considered here.
The continuous Langevin dynamics from \cref{eq:ODLangevin} then take the form
$$
    dX_t=-N\sigma^{-2}(X_t-\bar{y})dt+\sqrt{2}dW_t,
$$
so that the potential indeed takes the form of a finite sum as in \cref{eq:FiniteSum}, for $X \in \mathbb{R}^{d}$ defined by
$$
\nabla F(X)=\frac{1}{N}\sum_{i=1}^N N\sigma^{-2}(X-y_i).
$$
The stochastic gradient in \cref{eq:SGLD} is (cf. \cref{eq:stochgrad}) then generated, for $x\in \mathbb{R}^{d}$ by
$$
\widehat{\nabla F}(x)=N\sigma^{-2}(x-\widehat{y}_{\boldsymbol{\omega}_i}),\quad \widehat{y}_{\boldsymbol{\omega}_i}\equiv\frac{1}{n}\sum_{j=1}^{n}y_{\omega_{ij}},
$$
 with $n\ll N$, and the vectors $\boldsymbol{\omega}_{i}$ are generated via either the RR or RM randomisation procedure according to \cref{alg:GenAlg}. For ease in the following we will generically denote the stochastic estimate at iteration $k$ as $\widehat{y}_{k}$, without risk of confusion since the index is an integer not a vector.

The iterations for the SGLD scheme \cref{eq:SGLD} for the problem \cref{eq:ModelProblem} with initialisation $x_{0} \in \mathbb{R}$ are defined by the update rule
$$
x_{k+1}=x_k-hN\sigma^{-2}(x_k-\widehat{y}_k)+\sqrt{2h}\xi_k,\quad\xi_k\sim\mathcal{N}(0,1),
$$
which, upon rescaling time (preconditioning) $h\gets \sigma^2 h/N$ gives
\begin{equation}\label{eq:ModProbSGLD}
x_{k+1}=(1-h)x_k+h\widehat{y}_k+\sigma\sqrt{\frac{2h}{N}}\xi_k,\quad\xi_k\sim\mathcal{N}(0,1).
\end{equation}
After $k$ iterations, one has, starting from $x_0=0$,
\begin{equation}\label{eq:GenEx}
x_{k+1}=h\sum_{i=0}^k(1-h)^{i}\widehat{y}_{k-i}+\sigma\sqrt{\frac{2h}{N}}\sum_{i=0}^k(1-h)^i\xi_{k-i},
\end{equation}
therefore the asymptotic error in the mean is zero, since
$\E[x_{k}]=(1-(1-h)^{k})\bar{y}$ and so $\lim_{k\to\infty}\E[x_{k}]=\bar{y}$ \cite{Vollmer2016}. Consequently, to examine the differing asymptotic errors, we look at the asymptotic variance of the samples $\lim_{k\to\infty}\V[x_{k}]\equiv\V[x_{\infty}]$ and see how close it is to the  variance  $\sigma^2/N$ of the target $\pi_*$. 

{
The following calculations were performed using code available at the GitHub repository associated to this article and linked to in \cref{sec:Experiments}.
}

\medskip

\paragraph{ULA}
It is useful to compare the SGLD schemes with the scheme \cref{eq:GLD} (with no stochastic gradient bias) which inherits only the numerical bias of the Euler-Maruyama discretisation, which uses iterations of the form (after rescaling time $h\gets \sigma^2 h/N$)
\begin{equation}
    x_{k+1}=(1-h)x_k+h\bar{y}+\sigma\sqrt{\frac{2h}{N}}\xi_k,\quad\xi_k\sim\mathcal{N}(0,1).\label{eq:ModProbULA}
\end{equation}
For $x_0=0$, it is easy to see that $\lim_{k\to\infty}\E[x_{k}]=\bar{y}$ and that the asymptotic relative variance error follows $N\sigma^{-2}(\V[x_{\infty}]-\sigma^{2}/N)=2/(2-h)$ \cite{Vollmer2016}.

\medskip

\paragraph{SGLD-RM}
For RM, since the covariance between stochastic estimates $\widehat{y}_i,\widehat{y}_j$ is 0 except for $i=j$ with $V\equiv\mathbb{V}[\widehat{y}]=(nN)^{-1}(N-n)/(N-1)\left[\sum_{i=1}^N y_i^2-\bar{y}^2\right]$ \cite[Eq.(16)]{Vollmer2016}, one has directly
\begin{equation*}
\begin{aligned}
\V[x_{k+1}]&=h^2\sum_{i=0}^k(1-h)^{2i}V+\frac{2h\sigma^2}{N}\sum_{i=0}^k(1-h)^{2i}\\
&=Vh^2\frac{1-(1-h)^{2(k+1)}}{1-(1-h)^2}+2h\sigma^2\frac{(1-(1-h)^{2(k+1)})}{N(1-(1-h)^2)}.
\end{aligned}
\end{equation*}
The asymptotic limit $\V[x_{\infty}]=Vh/(2-h)+N^{-1}\sigma^2/(1-h/2),$ and
therefore the relative asymptotic bias in the variance is 
$$\frac{\V[x_{\infty}]-\sigma^{2}/N}{\sigma^{2}/N}=\frac{hNV}{\sigma^2(2-h)}+\frac{2}{2-h}-1=h\frac{NV}{2\sigma^2}+\frac{h}{2}+\mathcal{O}(h^2),$$
where $V\approx\sigma^2/n$ and hence the relative variance error is $\approx (hR+h)/2$ to leading order, in agreement with \cite{Vollmer2016}.

\medskip

\paragraph{SGLD-RR} 
For RR we begin by rewriting the expression \cref{eq:GenEx} as
\begin{gather*}
    x_{k+1}=h(1-h)^{r}\sum_{i=0}^{n_e-1}(1-h)^{iR}\sum_{j=0}^{R-1}(1-h)^{j}\widehat{y}^{(i,r)}_{k-j}
+h\sum_{i=0}^{r-1}(1-h)^{i}
\widehat{y}_{k-i}
+\sigma\sqrt{\frac{2h}{N}}\sum_{i=0}^k(1-h)^i\xi_{k-i}.
\end{gather*}
where $n_e=\lfloor(k+1)/R\rfloor, r=(k+1)-n_eR$ and $\widehat{y}^{(i,r)}_{k-j}=\widehat{y}_{k-j-iR-r}$.  Consequently, one has that the variance follows 
\begin{equation*}
\begin{aligned}
\V[x_{k+1}]=h^2(1-h)^{2r}\sum_{i=0}^{n_e-1}&(1-h)^{2iR}\sum_{j,j'=0}^{R-1}(1-h)^{(j+j')}
\text{cov}(\widehat{y}^{(i,r)}_{k-j},\widehat{y}^{(i,r)}_{k-j'})\\
&+h^2\sum_{i,i'=0}^{r-1}(1-h)^{(i+i')}
\text{cov}(\widehat{y}_{k-i},\widehat{y}_{k-i'})+\frac{2h\sigma^2}{N}\sum_{i=0}^k(1-h)^{2i}.
\end{aligned}
\end{equation*}

One has $\text{cov}(\widehat{y}_{j},\widehat{y}_{j'})=\E[(\widehat{y}_{j'}-\bar{y})\E[(\widehat{y}_{j}-\bar{y})|\widehat{y}_{j'}]]$ by the Law of Iterated Expectation. Since $\E[(\widehat{y}_{j}-\bar{y})|\widehat{y}_{j'}]=(N\bar{y}-n\widehat{y}_j)/(N-n)$, then in the spirit of \cite[Lemma 1]{Mishchenko2020}, one has
\begin{equation*}\label{eq:covVkk}
\text{cov}(\widehat{y}_{j},\widehat{y}_{j'})=-\frac{\V[\widehat{y}_{j}]}{R-1},
\end{equation*}
where $\V[\widehat{y}_{j}]=\text{cov}(\widehat{y}_{j},\widehat{y}_{j})=V$ is the same as the variance $V$ for RM above. One may then use the formulae for sums of geometric series, after separating into the $i\neq j$ and $i=j$ cases, to derive, taking $n_{e}\to\infty$,
\begin{gather*}
\frac{\V[x^{(r)}_{\infty}]-\sigma^{2}/N}{\sigma^{2}/N}=\frac{NV}{\sigma^2(R-1)}\left[\frac{Rh}{2-h}
-\left(\frac{(1-h)^{2r}(1-(1-h)^R)^2}{1-(1-h)^{2R}}+(1-(1-h)^r)^2\right)\right]+\frac{2}{2-h}-1,
\end{gather*}
where we introduce a different\footnote{A manifestation of the fact that SGLD-RR does \emph{not} give rise to a well-defined invariant distribution, in distinction to SGLD-RM.} asymptotic variance $\V[x^{(r)}_{\infty}]$ for each point in the epoch $r=0,\ldots R-1$, since variance error is periodic in $r$ (see \cref{fig:EMRR1}), a phenomenon which is familiar from the optimisation literature on SGD-RR \cite{Mishchenko2020}.
It is then possible to show that the average variance error over a period follows
$$
\frac{1}{R}\sum_{r=0}^{R-1}\frac{\V[x^{(r)}_{\infty}]-\sigma^{2}/N}{\sigma^{2}/N}=\frac{h}{2}+\frac{h^2}{4}+\frac{NVh^2(R+1)}{4\sigma^2}+\mathcal{O}(h^3),
$$
and so is $\mathcal{O}(h+(Rh)^2)$, which compares very favourably with SGLD-RM (see \cref{fig:EMRR2}).
\begin{figure}
\centering
\begin{subfigure}[t]{0.49\textwidth}
        \includegraphics[width=\textwidth]{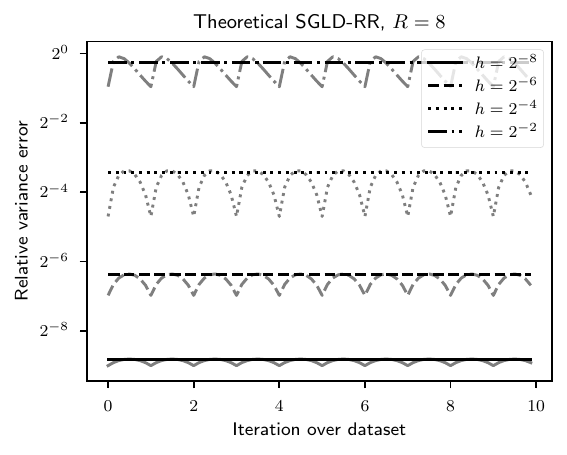}
        \caption{}\label{fig:EMRR1}
\end{subfigure}
\begin{subfigure}[t]{.49\textwidth}
\includegraphics[width=\textwidth]{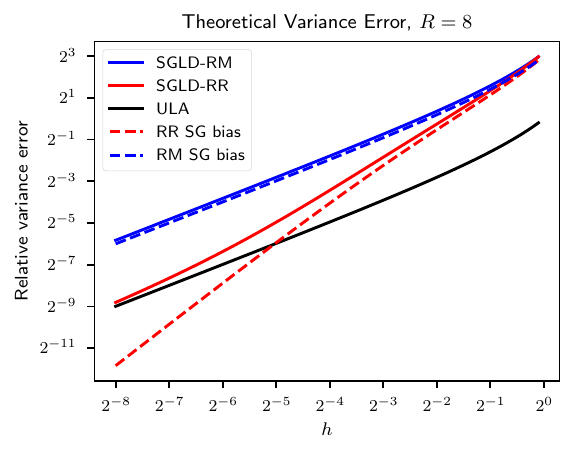}
\caption{}\label{fig:EMRR2}
\end{subfigure}
    \caption{a) For SGLD-RR the variance error is periodic, with period equal to an epoch. As $h$ decreases, the (non-periodic) bias due to using the Euler-Maruyama discretisation dominates the (periodic) gradient noise, and the oscillation amplitude decreases. Horizontal lines show the averaged variance error.
    b) The variance error expressions derived in \cref{sec:ModelProblem} for SGLD-RR and SGLD-RM compared. We set $V=R\sigma^2/N$. For RR, the overall error approaches the numerical bias (present even in the full gradient case of \cref{eq:GLD}). For SGLD-RM, the stochastic gradient bias dominates for all values of $h$.}
        \label{fig:SGLDRR}
\end{figure}

The theoretical results of the analysis shown in \cref{fig:SGLDRR} are confirmed by experiments (see \cref{fig:SGLDExperiment}).
\begin{figure}
    \centering
    \begin{subfigure}{.3\textwidth}
            \includegraphics[width=\textwidth]{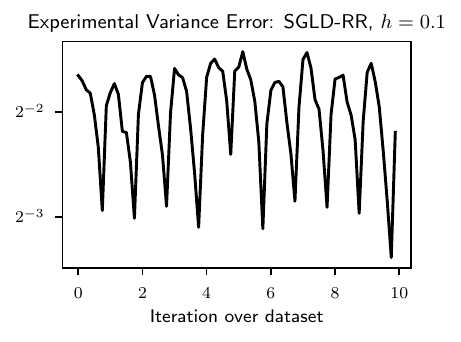}
        \caption{}
    \end{subfigure}
    \begin{subfigure}{.3\textwidth}
            \includegraphics[width=\textwidth]{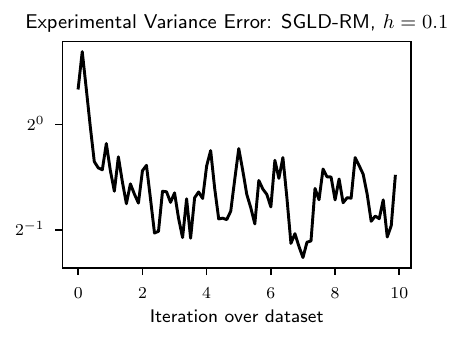}
            \caption{}
    \end{subfigure}
    \begin{subfigure}{.3\textwidth}
            \includegraphics[width=\textwidth]{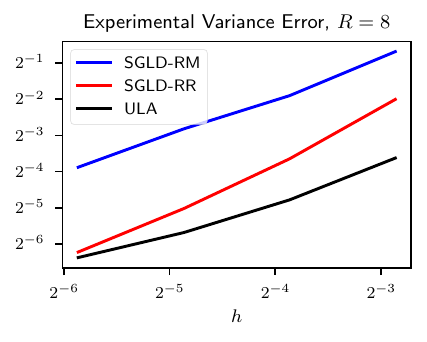}
            \caption{}
    \end{subfigure}
    \caption{Experiment for model problem \cref{eq:ModelProblem} with $\sigma^2=1$. We set $R=8$, $N=20\times R$, and draw $N$ samples $y_i$ from $\mathcal{N}(0,1)$. The number of epochs for sample collection is set as $n_{e}=100+20(hR)^{-3}$, with a burnin of 1000 steps. We then generate ULA samples via \cref{eq:ModProbULA}, while for SGLD-RR/SGLD-RM we use \cref{eq:ModProbSGLD} with the appropriate strategy and average over $10^4$ simultaneous stochastic gradient realisations. Plotted is the relative variance error $N\widehat{V}-1$ where we calculate the variance of the chain of samples as $\widehat{V}$ after a burnin of 1000 steps. Periodicity in the asymptotic regime (figures show last 10 epochs) is apparent for RR but not for RM as seen in (a),(b); in addition, the bias is reduced by switching from RM to RR, and the latter approaches the line for ULA (i.e. Euler-Maruyama discretisation using the full gradient).}
    \label{fig:SGLDExperiment}
\end{figure}

\subsection{Discussion}
Clearly, the target distribution $\pi_*=\mathcal{N}(\bar{y},\sigma^2/N)$ is Gaussian. The ULA iterates $x_{k}^{[ULA]}$ are also Gaussian for any $k$ (since they can be expressed as the sums of Gaussian variables). However the SGLD-RR and SGLD-RM iterates are not, as the stochastic gradients $\hat{y}$ are not Gaussian random variables. One could argue that, via the Central Limit Theorem, the approximation of them by Gaussian random variables is valid however \cite{Leimkuhler2016}. In this case one may use the formula for the Wasserstein distance between two 1D normal distributions $\mu_1=\mathcal{N}(m_1,\sigma^2_1),\mu_2=\mathcal{N}(m_2,\sigma^2_2)$
 $$
 \mathcal{W}_2(\mu_1,\mu_2)=\left[(m_1-m_2)^2+(\sigma_1-\sigma_2)^2\right]^{1/2},
 $$
 to calculate (approximate) asymptotic Wasserstein distances relative to the target $\pi_*=\mathcal{N}(\bar{y},\sigma^2/N)$ for the different schemes applied to the model problem. These distances depend only on the variance error, since for every scheme the asymptotic error in the mean is 0:
\begin{equation*}
    \begin{aligned}
        \mathcal{W}_2(\widetilde{\pi}_{\infty}^{[ULA]},\pi_*)&=\frac{\sigma}{\sqrt{N}}\left|1-\frac{1}{\sqrt{1-h/2}}\right|=\mathcal{O}(h);\\
        \mathcal{W}_2(\widetilde{\pi}_{\infty}^{[RM]},\pi_*)&\approx\frac{\sigma}{\sqrt{N}}\left|1-\sqrt{\frac{2+hNV/\sigma^2}{2-h}}\right|=\mathcal{O}(h+Rh); \\
\max_{0\leq r<R}\mathcal{W}_2(\widetilde{\pi}_{\infty,r}^{[RR]},\pi_*)&\approx\mathcal{O}(h+(Rh)^2),
    \end{aligned}
\end{equation*}
where we introduce $R$ distributions $\{\widetilde{\pi}_{\infty,r}^{[RR]}\}_{r=0}^{R-1}$ for each asymptotic variance $\V[x^{(r)}_{\infty}]$ derived for SGLD-RR above. We also use $\approx$ to indicate the use of the approximate Gaussianity of the stochastic gradients.

For the Gaussian case, one then has much better dependence on $Rh$ than would be expected from the bounds {\cref{eq:SGLDRM_ass3,eq:SGLDRR_ass3}} in \cref{thm:SGLDRM,thm:SGLDRR} (the potential for the model problem fulfills \cref{assum:additional_smoothness}). {Specifically, \cref{eq:SGLDRM_ass3} suggests convergence of $\mathcal{O}(h+(Rh)^{1/2})$ for SGLD-RM, worse than the $\mathcal{O}(h+Rh)$ for the Gaussian obtained above; \cref{eq:SGLDRR_ass3} suggests convergence of $\mathcal{O}(h+Rh)$ for SGLD-RR, worse than the $\mathcal{O}(h+(Rh)^2)$ for the Gaussian obtained above. For SGLD-RM, whether the Gaussian is simply special or actually suggests the bound in \cref{thm:SGLDRM} is loose is a well-known open problem in the field;} we can now add the analogous question for SGLD-RR. Related results obtained via different techniques for observables or the invariant distribution itself suggest higher order convergence than the best known Wasserstein bounds for SGLD-RM in \cref{thm:SGLDRM} \cite{Vollmer2016,Nagapetyan2017}.

\section{Logistic Regression Experiments}\footnote{Code for the experiments and for the analytical Gaussian calculations in \cref{sec:ModelProblem} may be found at \url{https://github.com/lshaw8317/RandomReshuffleSGLD}.}\label{sec:Experiments}
As a final demonstration we apply SGLD with the two randomisation strategies to {logistic regression problems for two real datasets (``StatLog'' and ``CTG'') and a simulated dataset (``SimData'') \cite{Casas2022}. The negative log-density is then, $X\in \mathbb{R}^{d}$,}
$$
F(X)=\sum_{i=1}^NX^TD^{-1}X -z_iX^T\widetilde{{y}}_i+\log\left[1+\exp\left(X^T\widetilde{{y}}_i\right)\right],
$$
where the datapoints $y_i=[z_i,\widetilde{y}_i]^T$ are composed of labels $z_i\in\{0,1\}$ 
and feature variables $\widetilde{{y}}_i\in\R^d$ for $i = 1,...,N$. The diagonal matrix  $D=\mathrm{diag}\{N\lambda_i^2:i=1,\ldots, d\}$ encodes the prior information{, with $\lambda_i^2=25, \forall i$}. This potential then fulfills \cref{assum:smoothness}\cref{assum:smoothnessB} and \cref{assum:additional_smoothness}.
{
\paragraph{SimData}
We generate simulated data according to the same procedure and parameter values described in \cite{Casas2022}.
The first step is to generate
${\hat{y}}_i\sim\mathcal{N}(0,{\sigma}^2)$ with ${\sigma}^2=\mathrm{diag}\left\{\sigma_j^2: j=1\ldots,d-1\right\}$, where
$$
\sigma^2_j=\begin{cases}
25 &  j\leq 5\\
1 &  5<j\leq 10\\
0.04 &  j>10
\end{cases}.
$$
Then, we generate the true parameters $X_{true}=[\alpha,{\beta}]^T$ with $\alpha\sim\mathcal{N}(0,\gamma^2)$ and the vector ${\beta}\in\mathbb{R}^{d-1}$ with independent components following $\beta_j\sim\mathcal{N}(0,\gamma^2),j=1,\ldots,d-1$, with $\gamma^2=1$.
Augmenting the data $\widetilde{{y}}_i=[1,\hat{y}_i^T]^T$, $z_i$ is then generated as a Bernoulli random variable $z_i\sim\mathcal{B}((1+\exp(-X_{true}^T\widetilde{{y}}_i))^{-1})$.
In concreteness, a simulated data set $\{z_i, \widetilde{y}_i\}_{i=1}^n$ with $n=1024$ samples is generated, $\hat{y}_i\in\mathbb{R}^{d-1}$ with $d-1=100$.
\paragraph{StatLog} This dataset is a common benchmark for logistic regression problems,  and has $n = 4435$, $d-1 = 36$.
\paragraph{CTG} This dataset is a common benchmark for logistic regression problems,  and has $n = 2126, d-1 = 21$.
}

We evaluate the true posterior mean $\mu\equiv\mathbb{E}_{X\sim\pi_*}[X]$  (established by a long HMC run of {$10^6$} samples) and examine convergence of the empirical mean via
$$
\frac{\left\|\Delta\mu\right\|}{\left\|\mu\right\|}\equiv\frac{\left\|\frac{1}{K}\sum_{k=1}^Kx_k-{\mu}\right\|_2}{\left\|\mu\right\|_2}.
$$
The empirical mean is calculated using samples $(x_k)_{k=1}^K$ generated via the SGLD-RR, SGLD-RM and ULA schemes. In \cref{fig:Exps}, one can compare the SGLD methods to the performance of the ULA scheme with full gradient, which has no stochastic gradient bias (only that associated with the discretisation).

We confirm the presence of oscillations for non-Gaussian targets, and the increased order of convergence to the posterior mean for RR compared to RM. We emphasise that RR actually runs faster for a given number of epochs owing to its exploitation of caching, which makes its absolute improved performance all the more remarkable, since it essentially comes at a discount\footnote{In \cite{Mishchenko2020}, the authors note that the RM procedure may be implemented in a slightly more efficient way than the procedure described in \cref{sec:randomisation_strategies}. One may permute the data only once, at the very beginning of sampling, and then at every epoch, randomly select $R$ dataset indices in $\{0,\ldots, N-1\}$ and for each such index take the succeeding $n=N/R$ datapoints to form each batch. This requires less cached memory and so runs faster, whilst giving the same results in expectation, since in expectation it iterates over all possible batches.}.

\begin{figure}
\centering
\begin{subfigure}{.3\textwidth}
\includegraphics[width=\textwidth]{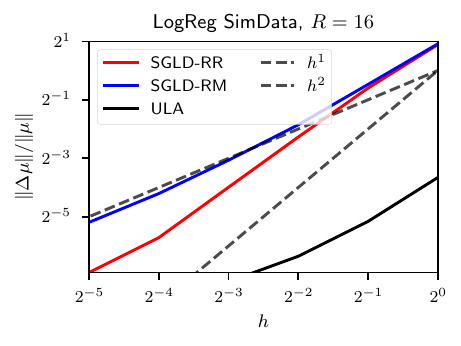}
\end{subfigure}
\begin{subfigure}{.3\textwidth}
\includegraphics[width=\textwidth]{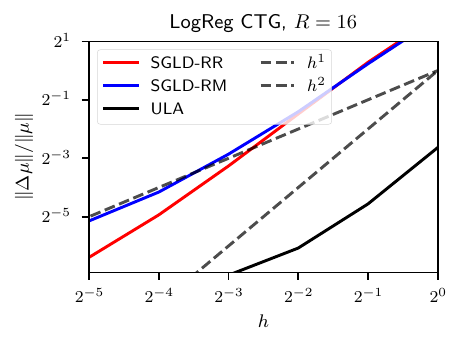}
\end{subfigure}
\begin{subfigure}{.3\textwidth}
\includegraphics[width=\textwidth]{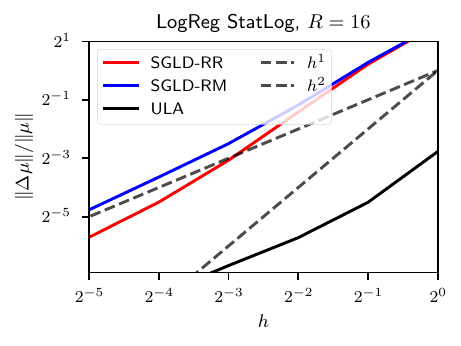}
\end{subfigure}
\begin{subfigure}{.3\textwidth}
\includegraphics[width=\textwidth]{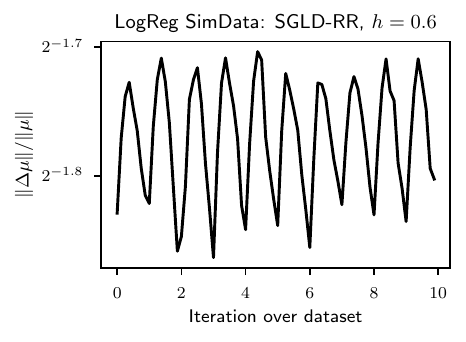}
\end{subfigure}
\begin{subfigure}{.3\textwidth}
\includegraphics[width=\textwidth]{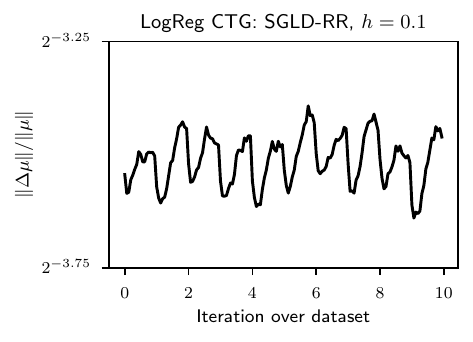}
\end{subfigure}
\begin{subfigure}{.3\textwidth}
\includegraphics[width=\textwidth]{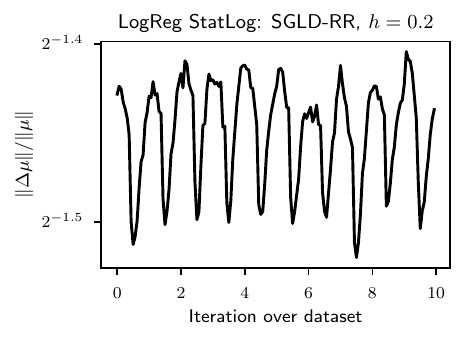}
\end{subfigure}
\caption{Experiments for SGLD with SGLD-RR, SGLD-RM and ULA, for some of the logistic regression problems from \cite{Casas2022} (although for SimData with a smaller dataset of size $N=1024$). The number of epochs $n_{e}=100+20(hR)^{-3}$, and we average over $10^4$ realisations. Periodicity for the relative error in the posterior mean in the asymptotic regime (figures show last 10 epochs) is apparent for RR but not for RM.
As a guide we include dashed reference lines of slope 1 and 2 for the error plots.}
\label{fig:Exps}
\end{figure}

\section{Conclusion}
As is now clear, Random Reshuffling is well-justified as an alternative batching strategy for the sampling problem just as is the case for optimisation, with lower overheads than the standard RM policy, in distinction to most other modifications to SGLD-RM which typically come with higher overheads (usually compensated for by better performance) \cite{Chatterji2018,Brosse2018}. It is perhaps of interest to explore the effect of other, more exotic batching strategies familiar from optimisation to see if they may also improve performance for sampling \cite{Cha2023}. We remark also that the RR procedure does not preclude the use of other variance reduction techniques - indeed we point out that the intermediate results of the proof of \cref{thm:SGLDRR} show that SGLD-RR is strong order one in the timestep $h$, which would allow for improved multilevel Monte Carlo and unbiased estimation methods \cite{GilesMultilevel,GilesActa,ZygalakisMultilevel,chada2023unbiased}.

As mentioned in the introduction, one does not have to limit oneself to using the Euler-Maruyama discretisation to derive the sampling scheme based on the Langevin equation \cref{eq:ODLangevin}. Given that the Gaussian analysis suggests that RR gives $\mathcal{O}(h^2)$ error in the gradient noise, it would be interesting to combine it with the Leimkuhler-Matthews discretisation \cite{LeMa13}, which is of weak order $2$. One would then have a method of overall order 2, with error $\mathcal{O}(h^2+(Rh)^2)$. 

While this work has focussed on machine learning applications, the principle of improved performance carries over to other fields such as molecular dynamics. It would be interesting to consider the implications of the results here for the Random-Batch Method (see \cite{RBM}), which is used for simulating systems of interacting particles, where one uses a stochastic gradient approximation based on a random subset of the particles. 

As final methodological point{s}, the most novel aspect of this work is the local discretisation analysis in \cref{prop:SGULA}, which can be {directly} combined with alternative Wasserstein convergence results in the non-convex setting to allow for theoretical guarantees without convexity \cite{DurmusBias2024,Eberle2019,Majka2020} {by instead assuming each $f_{\mathbf{\omega}}$ is convex outside of a ball of radius $R > 0$}. {In particular, one can consider the same interpolation argument of \cref{thm:SGLDRR} with an appropriately designed distance metric.}

{Further, the nature of the proofs in the optimisation setting in \cite{Mishchenko2020} extend to other minibatching strategies, namely, ``Shuffle-Once" and ``Incremental Gradient", where the minibatches and the order of the minibatches are fixed either deterministically or chosen once randomly at initialisation. We remark that the techniques developed here for analysing RR can also be extended to these strategies; however, yielding a worse dependence on $R$.}

\section*{Acknowledgements} The authors thank Yuansi Chen, Ben Leimkuhler, Daniel Paulin, Jes\'{u}s Mar\'{i}a Sanz-Serna and Kostas Zygalakis for helpful discussions in the early stages of this work. 

\appendix
\section{Proofs}\label{sec:proofs}
\begin{proposition}\label{prop:sg_contraction}
    Let us consider two realisations $(u_{i})^{\infty}_{i=0}$, $(v_{i})^{\infty}_{i=0}$ of SGLD-RR implemented via \cref{eq:SGLD} on $\mathbb{R}^{d}$ with step size $h>0$, objective function $F$, and shared noise, where $F:\mathbb{R}^{d} \to \mathbb{R}$ satisfies \cref{assum:smoothness}\cref{assum:smoothnessB}.  Assume that $h < 2/L$ then for $r,k\in \mathbb{N}$ we have
    \begin{equation}
        \|u_{r+k}-v_{r+k}\|_{L^{2}} \leq (1-h\mu)^{k/2}\|u_{r}-v_{r}\|_{L^{2}},
    \end{equation}
    where $\|\cdot\|_{L^{2}} := \left(\mathbb{E}\|\cdot\|^{2}\right)^{1/2}$.
\end{proposition}
\begin{proof}
Follows from \cite[Lemma 1]{Dalalyan2017further} or \cite[Lemma 2]{DalalyanSG}.
\end{proof}

\begin{proposition}\label{prop:SGULA}
Under the assumptions of \cref{thm:SGLDRR}, assuming that for $r\in \mathbb{N}$, $x_{r} = X_{r} \sim \pi_{*}$ we have for $k \in \mathbb{N}$
\begin{align*}
    &\|x_{r+k} - X_{(r+k)h}\|_{L^{2}} \leq \\
    &4e^{khL}\Bigg(h\sqrt{R}\sigma_{*}  + h^{3/2}L\sqrt{k}(R\sqrt{hLd} + \sqrt{2Rd})+\frac{1}{2}kh^2\left(L\sqrt{Ld} +L_{1}d\right)+ \frac{1}{2}h^{3/2}L\sqrt{kd}\Bigg).
\end{align*}
If \cref{assum:additional_smoothness} does not hold then instead we can bound
\begin{align*}
    &\|x_{r+k} - X_{(r+k)h}\|_{L^{2}} \leq \\4e^{khL}\Bigg(h\sqrt{R}\sigma_{*}  + &h^{3/2}L\sqrt{k}(R\sqrt{hLd} + \sqrt{2Rd})+ \frac{1}{2}kh^{3/2}\left(h^{1/2}L\sqrt{Ld} + L\sqrt{d}\right)\Bigg).
\end{align*}
\end{proposition}

\begin{proof}
We introduce the notation $\omega_{k}$ for $k \in \mathbb{N}$ to denote the vector $\boldsymbol{\omega}$ of batch indices at iteration $k$, which is defined for Random Reshuffling by the procedure introduced in \cref{sec:randomisation_strategies}. If we consider the difference between the continuous dynamics and stochastic gradient scheme at iteration $k$, $\Delta_{k} := x_{k}-X_{kh}$, we have that
 \begin{align}\label{eq:def_EiDi}
\Delta_{k} &= \Delta_{k-1} - \int^{kh}_{(k-1)h}(\nabla f_{\omega_{k-1}}(x_{k-1}) - \nabla F(X_t ))dt\nonumber\\
&=  \Delta_{k-1} - h\underbrace{(\nabla f_{\omega_{k-1}}(x_{k-1}) - \nabla f_{\omega_{k-1}}(X_{(k-1)h}))}_{Q_{k-1}} \nonumber\\
&+ h\underbrace{\left[ \nabla F(X_{(k-1)h})- \nabla f_{\omega_{k-1}}(X_{(k-1)h})\right]}_{:= D_{k-1},\textnormal{ stochastic gradient error}} + \underbrace{\int^{kh}_{(k-1)h}(\nabla F(X_{t})-\nabla F(X_{(k-1)h}))dt}_{:= E_{k-1}, \textnormal{ local error of full gradient scheme}},
\end{align}
where for each $k\in \mathbb{N}$, $\|Q_{k}\|_{L^{2}} \leq L\|\Delta_{k}\|_{L^{2}}$ and using Taylor's theorem each $E_{k}$ can be written as 
$
    E_{k} = A^{k}_{2} + A^{k}_{3/2},
$
where 
\begin{align*}
&A^{k}_{3/2} := \sqrt{2}\int^{kh}_{(k-1)h}\nabla^{2}F(X_{(k-1)h})\int^{t}_{(k-1)h}dW_{s'}dt,\\
    &A^{k}_{2} := \int^{kh}_{(k-1)h}\int^{1}_{0}\nabla^{2}F(X_{t}+s(X_{(k-1)h} - X_{t}))ds\Bigg[-\int^{t}_{(k-1)h}\nabla F(X_{s'})ds'\\
    &+ \sqrt{2}\int^{t}_{(k-1)h}dW_{s'}\Bigg]dt - A^{k}_{3/2} ,
\end{align*}
and for any $k \in \mathbb{N}$, $\|A^{k}_{2}\|_{L^{2}} \leq h^{2}\left(L\sqrt{Ld} +2L_{1}d \right)$ and $\|A^{k}_{3/2}\|_{L^{2}} \leq h^{3/2}L\sqrt{d}$, which can be shown using \cref{assum:smoothness}\cref{assum:smoothnessA}, \cref{assum:additional_smoothness}, \cref{lemma:disc_error}, and \cite[Lemma 2]{Dalalyan2017further} as follows
\begin{align*}
    &\left\|A^{k}_{3/2}\right\|_{L^{2}} \leq \sqrt{2}L\int^{kh}_{(k-1)h}\left\|\int^{t}_{(k-1)h}dW_{s'}\right\|_{L^{2}}dt \leq h^{3/2}L\sqrt{d},\\
    &\left\|A^{k}_{2}\right\|_{L^{2}} \leq  \left\|\int^{kh}_{(k-1)h}\int^{1}_{0}\nabla^{2}F(X_{t}+s(X_{(k-1)h} - X_{t}))ds\int^{t}_{(k-1)h}\nabla F(X_{s'})ds'dt\right\|_{L^{2}}\\
    &+ \sqrt{2} \int^{kh}_{(k-1)h}\int^{1}_{0}\left\|\left[\nabla^{2}F(X_{t}+s(X_{(k-1)h} - X_{t})) - \nabla^{2}F(X_{(k-1)h})\right] \int^{t}_{(k-1)h}dW_{s'}\right\|_{L^{2}}dsdt,\\
    \intertext{{where we now use \cref{assum:smoothness}, \cite[Lemma 2]{Dalalyan2017further} and \cref{assum:additional_smoothness} successively to achieve the bound}}
    &\leq \frac{h^{2}}{2}L\sqrt{Ld} + \sqrt{2}L_{1} \int^{kh}_{(k-1)h}\left\|\frac{1}{2}\|X_{t} - X_{(k-1)h}\|_{2}\times \left\|\int^{t}_{(k-1)h}dW_{s'}\right\|_{2}\right\|_{L^{2}}dt\\
    &\leq \frac{h^{2}}{2}L\sqrt{Ld} + \frac{\sqrt{2}}{4}L_{1} \int^{kh}_{(k-1)h}\|X_{t} - X_{(k-1)h}\|^{2}_{L^{2}} +  \left\|\int^{t}_{(k-1)h}dW_{s'}\right\|^{2}_{L^{2}}dt\\
    &\leq \frac{h^{2}}{2}L\sqrt{Ld} + \frac{\sqrt{2}}{4}L_{1} \int^{kh}_{(k-1)h}\|X_{t} - X_{(k-1)h}\|^{2}_{L^{2}} +  \left\|\int^{t}_{(k-1)h}dW_{s'}\right\|^{2}_{L^{2}}dt\\
    &\leq \frac{h^{2}}{2}L\sqrt{Ld} + \frac{\sqrt{2}}{4}L_{1}\int^{h}_{0}\left[2t^2 Ld + 4td + td \right]dt\\
    &\leq h^{2}\left(L\sqrt{Ld} +2L_{1}d \right).
\end{align*}
Further, we have $\mathbb{E}[A^{i}_{3/2}A^{j}_{3/2}] = 0$, for $i \neq j$ and $i,j \in \mathbb{N}$, consequently $\|\sum^{i}_{k=1}A^{k}_{3/2}\|_{L^{2}} \leq \sqrt{i}\max_{1\leq k\leq i}\|A^{k}_{3/2}\|_{L^{2}}$. These properties are used to only lose a factor of $h^{1/2}$ when going from local error to global error in the same spirit as classical analysis for the strong convergence of numerical solutions of SDEs (see \cite{sanz2021wasserstein} and \cite{MiTrBook}). If  \cref{assum:additional_smoothness} does not hold then we can bound each $E_{k}$ by $\|E_{k}\|_{L^{2}} \leq h^{2}L\sqrt{Ld} + \sqrt{2}h^{3/2}L\sqrt{d}$ using the fact that $\nabla F$ is $L$-Lipschitz and \cref{lemma:disc_error}.
   
Iteratively applying the previous arguments we have for $r,k \in \mathbb{N}$ such that $\Delta_{r} = 0$ the following relation holds
\begin{align*}
    \Delta_{r+k} &= \sum^{r+k}_{i=r+1}[\Delta_{i}-\Delta_{i-1}] =  -\sum^{r+k-1}_{i=r}hQ_{i} + \sum^{r+k-1}_{i=r}\left(hD_{i} + E_{i}\right).
\end{align*}
We also have using that for $l\in \mathbb{N}$
\begin{align*}
    &h\left\|\sum^{R-1}_{i=0}D_{i+Rl}\right\|_{L^{2}} = \\
    &h\Bigg\|\sum^{R-1}_{i=0}\nabla f_{\omega_{i + Rl}}(X_{(i + Rl)h}) - \nabla f_{\omega_{i + Rl}}(X_{Rlh}) + \nabla F(X_{Rlh}) - \nabla F(X_{(i + Rl)h})\Bigg\|_{L^{2}},
    \end{align*}
    which, upon applying \cite[Lemma 14]{paulin2024sampling}, gives
    \begin{align*}
    &\leq h\sqrt{\frac{7}{2}R}\max_{0 \leq i\leq R-1}\Bigg\|\nabla f_{\omega_{i + Rl}}(X_{(i + Rl)h}) - \nabla f_{\omega_{i + Rl}}(X_{Rlh}) + \nabla F(X_{Rlh}) - \nabla F(X_{(i + Rl)h})\Bigg\|_{L^{2}}\\
    &\leq 2h\sqrt{\frac{7}{2}R}L\max_{0 \leq i\leq R-1}\Bigg\|X_{(i + Rl)h} - X_{Rlh}\Bigg\|_{L^{2}}\\
    &\leq 2\sqrt{\frac{7}{2}}hL\left(hR^{3/2}\sqrt{Ld} + R\sqrt{2hd}\right),
\end{align*}
where the final inequality follows from \cref{lemma:disc_error}. 
Similarly, applying \cite[Lemma 14]{paulin2024sampling} directly to $h\left\|\sum^{R-1}_{i=0}\nabla f_{\omega_{i + Rl}}(X_{(i + Rl)h}) - \nabla F(X_{(i + Rl)h})\right\|_{L^{2}}$ one has $h\left\|\sum^{R-1}_{i=0}D_{i+Rl}\right\|_{L^{2}}\leq h\sqrt{7R/2}\sigma_*$.

Using the preceding bound we have (using: a) the sum convention that, if the lower limit is less than the upper limit, the sum is empty; and b) that the maximum over zero elements is 0)
\begin{align*}
    &\|\Delta_{r+k}\|_{L^{2}} \leq hL\sum^{r+k-1}_{i=r}\left\|\Delta_{i}\right\|_{L^{2}} + \underbrace{h\left\|\sum^{r+k-1}_{i=R\lfloor (r+k-1)/R\rfloor }D_{i} \right\|_{L^{2}}}_{\textnormal{tail terms}}+ \underbrace{h\left\|\sum^{R\lfloor (r+k-1)/R\rfloor - 1}_{i=R\lceil r/R\rceil}D_{i}\right\|_{L^{2}}}_{\textnormal{fully contained epochs}}\\
    &+ \underbrace{h\left\|\sum^{R\lceil r/R\rceil -1}_{i=r}D_{i}\right\|_{L^{2}}}_{\textnormal{head terms}} + \underbrace{\left\|\sum^{r+k-1}_{i=r}E_{i}\right\|_{L^{2}}}_{\textnormal{sum of local errors}}\\
    &\leq hL\sum^{r+k-1}_{i=r}\left\|\Delta_{i}\right\|_{L^{2}} + h\left\|\sum^{r+k-1}_{i=R\lfloor (r+k-1)/R\rfloor }D_{i} \right\|_{L^{2}}+ h\left\|\sum^{R\lceil r/R\rceil -1}_{i=r}D_{k}\right\|_{L^{2}}\\
    &+h\sqrt{\lfloor (r+k-1)/R\rfloor - \lceil r/R\rceil}\max_{\lceil r/R\rceil \leq l < \lfloor (r+k-1)/R\rfloor}\left\|\sum^{(l+1)R - 1}_{i=lR}D_{i}\right\|_{L^{2}}\\
    &+ k\max_{r\leq i \leq r+k-1}\left\|A^{i}_{2}\right\|_{L^{2}} + \sqrt{k}\max_{r\leq i \leq r+k-1}\left\|A^{i}_{3/2}\right\|_{L^{2}}\\
    &\leq hL\sum^{r+k-1}_{i=r}\|\Delta_{i}\|_{L^{2}} + \underbrace{2\sqrt{\frac{7}{2}}h\sqrt{R}\sigma_{*}}_{\textnormal{bound on head and tail terms}}  + \underbrace{2\sqrt{\frac{7}{2}}h\sqrt{k}L(hR\sqrt{Ld} + \sqrt{2hRd})}_{\textnormal{bound on fully contained epochs}} \\
    &+ \underbrace{2h^2 k \left(L\sqrt{Ld} +L_{1}d\right)+ h^{3/2}L\sqrt{2kd}}_{\textnormal{bound on sum of local errors}}\\
    &\leq 2e^{khL}\Bigg(\sqrt{\frac{7}{2}}h\sqrt{R}\sigma_{*}  + \sqrt{\frac{7}{2}}h\sqrt{k}L(hR\sqrt{Ld} + \sqrt{2hRd})+ h^2 k \left(L\sqrt{Ld} +L_{1}d\right)+ h^{3/2}L\sqrt{kd}\Bigg),
\end{align*}
where in the second inequality we have used independence of the stochastic gradient noise from separate epochs. That is, for indices $i,j$ such that $D_i$ and $D_j$ belong to different epochs, the random variables corresponding to the stochastic gradient noise are independent - as a consequence, for $i>j$, $\E[D_{i}D_{j}]= \E\left[\E[D_iD_j|\omega_j]\right]=0$. Finally, if \cref{assum:additional_smoothness} does not hold one replaces the bound on the sum of local errors by $k\left(h^{2}L\sqrt{Ld} + \sqrt{2}h^{3/2}L\sqrt{d}\right)$ {, which follows from the fact that we can bound each $\|E_{i}\|_{L^{2}} \leq h^{2}L\sqrt{Ld} + \sqrt{2}h^{3/2}L\sqrt{d}$ (recalling that $E_i$ is defined in \cref{eq:def_EiDi}) using that $\nabla F$ is $L$-Lipschitz and \cref{lemma:disc_error}}.
\end{proof}

\begin{proof}[Proof of \cref{thm:SGLDRR}]
Inspired by the interpolation argument used in \cite{LePaWh24} and \cite{schuh2024convergence}, define $(x_{i})^{k}_{i=0}$ as $k \in \mathbb{N}$ steps of the stochastic gradient scheme and let $(X_{ih})^{k}_{i=0}$ be defined by \cref{eq:ODLangevin} at times $ih \geq 0$, where these are both initialised at $X_{0} = x_{0} \sim \pi_{*}$ and have synchronously coupled Brownian motion. We further define a sequence of interpolating variants $\mathbf{X}^{(l)}_{k}$ for every $l = 0,...,k$ all initialised $\mathbf{X}^{(l)}_{0} = x_{0}$, where we define $(\mathbf{X}^{(l)}_{i})^{l}_{i=1} := (X_{ih})^{l}_{i=1}$ and  $(\mathbf{X}^{(l)}_{i})^{k}_{i=l+1}$ by $k-(l+1)$ stochastic gradient integrator steps with coupled noise to $(x_{i})^{k}_{i=l+1}$. For $l = k$ we simply have just the continuous diffusion \cref{eq:ODLangevin}. Using  \cref{prop:SGULA} we split up the steps into blocks of size $\Tilde{k}$ as 
    \begin{align*}
        \|x_{k} - X_{kh}\|_{L^{2}}&\leq 
        \|\mathbf{X}^{(\lfloor k/\Tilde{k}\rfloor \Tilde{k})}_{k} - \mathbf{X}^{(k)}_{k}\|_{L^{2}} + \sum^{\lfloor k/\Tilde{k}\rfloor - 1}_{j=0}\|\mathbf{X}^{(j\Tilde{k})}_{k} - \mathbf{X}^{((j+1)\Tilde{k})}_{k}\|_{L^{2}}.
    \end{align*}
    We now use the fact that the continuous dynamics preserves the invariant measure and we have contraction of $k-l$ steps of the stochastic gradient integrator by \cref{prop:sg_contraction}. Then we have by considering each term in the previous summation 
    \begin{align*}
    &\|\mathbf{X}^{(j\Tilde{k})}_{k} - \mathbf{X}^{((j+1)\Tilde{k})}_{k}\|_{L^{2}} \leq \mathbf{A}\left(1-h\mu\right)^{k- (j+1)\Tilde{k}},
    \end{align*}
    where $\mathbf{A}$ depends on $\Tilde{k}$ as defined by the right hand side in the inequality of  \cref{prop:SGULA}. Summing up the terms we get
    \begin{align*}
     &\|x_{k} - X_{kh}\|_{L^{2}} \leq \frac{2\mathbf{A}}{1 - \left(1-h\mu\right)^{\Tilde{k}}} \leq \frac{2 \mathbf{A}}{1 - e^{-h\mu\Tilde{k}}} \leq 2\mathbf{A}\left(1 + \frac{1}{h\mu\Tilde{k}}\right).
    \end{align*}
   Let $\Tilde{k} = \left\lceil \frac{1}{hL} \right\rceil \leq \frac{2}{hL}$, then $\Tilde{k} \geq \frac{1}{hL}$ and
    \begin{align*}
    &\|x_{k} - X_{kh}\|_{L^{2}} \leq 4\mathbf{A} \frac{L}{\mu}\\
    &\leq 240h\frac{L}{\mu}\Bigg(\sqrt{R}\sigma_{*}  + (LR\sqrt{hd} + \sqrt{LRd})+\left(\sqrt{Ld} +\frac{L_{1}}{L}d\right)\Bigg),
    \end{align*}
    as required. Then combining this with \cref{prop:sg_contraction}, using a triangle inequality one can bound the 2-Wasserstein distance between $\widetilde{\pi}_{k}$ and $\pi_{*}$ by the $L^{2}$ distance between $k$ steps of the stochastic gradient scheme initialised at $\widetilde{\pi}_{0}$, $k$ steps of the stochastic gradient scheme initialised at $\pi_{*}$ and the continuous dynamics initialised at $\pi_{*}$ run for time $kh>0$, all with shared noise. 
\end{proof}

\begin{lemma}\label{lemma:disc_error}
    If $(X_{t})_{t \geq 0}$ is the solution to \cref{eq:ODLangevin} such that $X_{0} \sim \pi_{*}$. Assuming that $F$ is $L$-smooth we have for $t \geq 0$
    \[
    \|X_{t} - X_{0}\|_{L^{2}} \leq t\sqrt{Ld} + \sqrt{2td}.
    \]
\end{lemma}
\begin{proof}
We have
\begin{align*}
    \|X_{t} - X_{0}\|_{L^{2}} &= \left\|\int^{t}_{0}\nabla F(X_{s})ds + \sqrt{2}\int^{t}_{0}dW_{s}\right\|_{L^{2}}\\
    &\leq t\sqrt{Ld} + \sqrt{2td}
\end{align*}
using that $\|\nabla F(X_{s})\|_{L^{2}} \leq \sqrt{Ld}$ {which follows from \cite[Lemma 2]{Dalalyan2017further} and the fact that $X_{s} \sim \pi_{*}$ for any $s\in[0,t]$}.
\end{proof}
\newpage
\bibliographystyle{siam}

\end{document}